\documentclass[10pt]{amsart}
\usepackage{amssymb,amscd, amsfonts, amsthm}
\newtheorem{theorem}{Theorem}[section]

\newtheorem{corollary}[theorem]{Corollary}

\newtheorem{definition-lemma}[theorem]{Definition-Lemma}
\theoremstyle{definition}
\newtheorem{definition}[theorem]{Definition}

\theoremstyle{remark}
\newtheorem{remark}[theorem]{Remark}

\title[Supersingular Abelian Varieties]{ Characteristic Polynomial of  Supersingular Abelian Varieties over Finite Fields}
\author{Vijaykumar Singh, Gary McGuire, Alexey Zaytsev}
\begin{document}
\maketitle

\begin{abstract}  In this article, we give a complete description of the characteristic polynomials of supersingular abelian varieties over finite fields. We list them for the dimensions upto  $7$.
\end{abstract}
\section{Introduction}
Supersingular abelian varieties have applications in cryptography  and coding theory and related areas. Identity based ecryption and computation of weights of some Reed Muller codes are some of them. An important (isogeny) invariant which carries most of the information about supersingular curves is the characteristic polynomial of the Frobenius endomorphism.  Here, we give a complete description of the characteristic polynomials of supersingular abelian varieties over finite fields. We list them for the dimensions upto  $7$.

\section{Background}

\begin{definition}
Let $A$ be an abelian variety of dimension $g$ over $\mathbb{F}_q$, where $q =p^n$. For $l \neq p$, \emph{the characteristic polynomial} of Frobenius endomorphism $\pi_A$ is defined as
\begin{displaymath}
P_A(X):=\det(\pi_A-XId|_{V_l(A)}),
\end{displaymath}

\end{definition}
The above definition is independent of choice of $l$. The coefficients of  $P_A(X)$ are in $\mathbb{Z}$. Moreover, $P_A(X)$ can be represented as  
\begin{displaymath}
P_A(X) = X^{2g} + a_{1}X^{2g-1} + \cdots+a_iX^{2g-i}+ \cdots+ a_gX^g +qa_{g-1}X^{g-1}+\cdots+q^{g-i}a_iX^i+ \cdots + q^g.
\end{displaymath}

\begin{definition}
A \emph{Weil-$q$-number} $\pi$ is defined to be an algebraic number such that, for every embedding
$\sigma:\mathbb{Q}[\pi] \hookrightarrow \mathbb{C}$, $|\sigma \pi|=q^\frac{1}{2}$ holds.
\end{definition}
For $q=p^n$, the real Weil polynomials are $X^2-q$ when $n$ is odd and $X-\sqrt{q}$ when $n$ is even. Otherwise, the set of  roots of $P_A(X)$ has the form $\{\omega_1,\overline{\omega_1},\dots ,\omega_g, \overline{\omega_g}\}$, where the $\omega_i$'s  are Weil-$q$-numbers. A monic polynomial with integer coefficients which satisfies this condition is called a \emph{Weil polynomial}. Thus, every Weil polynomial of degree $2g$ has the form
above for certain integers $a_i \in \mathbb{Z}$. The converse is false; indeed, since the absolute value of the roots of $P(X)$ is prescribed (equal to $\sqrt q$), its coefficients have to be bounded.

An abelian variety $A$ is \emph{$k$-simple} if it is not isogenous to a product of abelian varieties of lower dimensions over $k$.
 In that case, $P_A(X)$ is either irreducible over $\mathbb{Z}$ or $P_A(X)=h(X)^e$, where $h(X)\in \mathbb{Z}[X]$ is an irreducible, monic polynomial over $\mathbb{Z}$; see Milne and Waterhouse in \cite{Waterhouse}.
 We have the following result from Tate \cite{Tate}.
 \begin{theorem} \label{Tate}Let $A$ and $B$ be abelian varieties defined over $\mathbb{F}_q$. Then $A$ is $\mathbb{F}_q$-isogenous to an abelian subvariety of $B$ if and only if $P_A(X)$ divides $P_B(X)$ in $\mathbb{Q}[X]$. In particular, $P_A(X) = P_B(X)$ if and only if $A$ and $B$ are $\mathbb{F}_q$-isogenous.
 \end{theorem}

Let $W(q)$ be the set of Weil-$q$-number in $\mathbb{C}$. We say that
two elements $\pi$ and $\pi^{'}$ are \emph{conjugate}, if $\pi$ and $\pi^{'}$ have the same minimum polynomial (called \emph{Weil polynomial}) over $\mathbb{Q}$.

The conjugacy class of $\pi_A$ depends only on the isogeny class of $A$, more precisely, we have the following theorem from \cite{J.Tate}.

\begin{theorem}\emph{(Honda-Tate theorem)}\label{H-T}
 The map $A \rightarrow \pi_A$ defines a bijection
\begin{displaymath} 
\{simple~abelian~varieties/ \mathbb{F}_q\}/(isogeny)\mapsto W(q)/ (conjugacy).
\end{displaymath}
\end{theorem}
In otherwords, given an irreducible Weil polynomial $P(X)$, there exist an unique abelian variety upto isogeny and $e$ a natural number such that $P_A(X)=P(X)^e$.

A Weil-$q$ number $\pi$ is called a \emph{supersingular Weil-$q$-number} if $\pi=\sqrt{q}\zeta$, where $\zeta$ is some root of unity.

\begin{definition}An abelian variety $A$ over $k$ is  called \emph{supersingular} if any one of the following (equivalent)
conditions holds:(see~\cite[Theorem~4.2]{oort1974subvarieties}) 
\begin{enumerate}
\item the eigenvalues of the Frobenius $\pi_A$ are supersingular Weil-$q$-numbers;
\item the Newton polygon of $A$ is a straight line of slope $1/2$;
\item $A$ is $\overline{k}$-isogenous to a power of a supersingular elliptic curve.
\end{enumerate}
\end{definition}

We have the following useful theorem from \cite{J.Tate}.
\begin{theorem}\label{dim} Let A be a simple abelian variety over $k = \mathbb{F}_q$, then
\begin{enumerate}
\item $End_k (A)\otimes \mathbb{Q}$ is a division algebra with center $\mathbb{Q}(\pi_A)$ and
                 
\begin{displaymath}
2dim A = [End_k (A)\otimes \mathbb{Q}:\mathbb{Q}(\pi_A)]^\frac{1}{2} [\mathbb{Q}(\pi_A) : \mathbb{Q}] .
\end{displaymath}
\item The division algebra $End_k (A)\otimes \mathbb{Q}$ over $\mathbb{Q}(\pi_A)$ has the following splitting
behaviour\\
\begin{enumerate}
\item it splits at each divisor $\mathfrak{l}$ of $l$ in $\mathbb{Q}(\pi_A)$, if $l\neq p$,
\item the invariants at the divisors  $\mathfrak{p}$ of $p$ in  $\mathbb{Q}(\pi_A)$ can be evaluated with
\begin{displaymath}                                       
  inv_{\mathfrak{p}} (End_k (A)\otimes \mathbb{Q}) \equiv \frac{v_\mathfrak{p}(\pi_A)}{v_\mathfrak{p}(q)} [\mathbb{Q}(\pi_A)_\mathfrak{p} : \mathbb{Q}_p ] \mod \mathbb{Z}, 
 \end{displaymath}              
 where $\mathbb{Q}(\pi_A)_\mathfrak{p}$ denotes the completion of $\mathbb{Q}(\pi_A)$ at $\mathfrak{p}$ and $v_\mathfrak{p}$ denotes the $\mathfrak{p}$-adic valuation.                           
\item it does not split at the real places of $\mathbb{Q}(\pi_A)$.
\end{enumerate}
\end{enumerate}
\end{theorem}
If $A$ is a supersingular abelian variety, then $\pi_A = \sqrt{q}\zeta$, where $\zeta$ is some root of unity. If $A$ is also simple, then it follows from above theorem that  $inv_{\mathfrak{p}} (End_k (A)\otimes \mathbb{Q}) \equiv \frac{1}{2}e_if_i \mod \mathbb{Z}$, where $e_i$ is the ramification index of $\mathfrak{p}$ over $p$ and $f_i$ is the residue class degree.\\
Also,  $ \pi \in \overline{\mathbb{Q}}$ be the supersingular Weil-$q$-number and  $P(X)$ be the corresponding minimal Weil polynomial. The invariants of $End_k (A)\otimes \mathbb{Q}$ lie in $\mathbb{Q}/\mathbb{Z}$. They can be evaluated from the
minimal polynomial $P(X)$ of $\pi_A$ as follows.
The only real Weil numbers are $q^{1/2}$ and -$q^{1/2}$ , so there are hardly
any real places of $\mathbb{Q}(\pi_A)$. We consider the polynomial $P(X)$ in $\mathbb{Q}_p [X]$, i.e., over the $p$-adic numbers. Let
                                    
\begin{displaymath}
P (X) = \displaystyle\prod_{i}f_i (X)
\end{displaymath}                                      
be the decomposition in irreducible factors in $\mathbb{Q}_p [X]$. The factors $f_i(X)$ correspond uniquely to the divisors $\mathfrak{p}_i$ of $p$ in $\mathbb{Q}(\pi_A )$. 
So to get the invariants, we have to factor $P(X)$ over $\mathbb{Q}_p$. Indeed,
\begin{eqnarray*}\nonumber
inv_{\mathfrak{p}_i} (End_k (A)\otimes \mathbb{Q}) \equiv \frac{v_p (f_i (0))}{ v_p (q)}  \mod \mathbb{Z}.
\end{eqnarray*}\nonumber
If $P(X)$ is a supersingular Weil polynomial, then $f_i (0)=\displaystyle\prod_{j=1}^{\deg f_i} \sqrt{q}\zeta_j$, where $\zeta_j$ is some root of unity.
In that case, $inv_{\mathfrak{p}_i} (End_k (A)\otimes \mathbb{Q})= \displaystyle\frac{n\frac{\deg f_i}{2}}{n} =\frac{\deg{f_i}}{2} \mod \mathbb{Z}.$

 Therefore the order of invariants is either $1$ or $2$ in $\mathbb{Q}/\mathbb{Z}$ depending on whether $\deg f_i$ is even or odd respectively. We use the invariants in order to evaluate the dimension of A as follows.
 
Since $End_k (A)\otimes \mathbb{Q}$ is a division algebra and  $\mathbb{Q}(\pi_A)$ is a number field, the  number $[End_k (A)\otimes \mathbb{Q}:\mathbb{Q}(\pi_A)]^\frac{1}{2}$ is equal to the order of
$End_k (A)\otimes \mathbb{Q}$ in the Brauer group of $\mathbb{Q}(\pi_A)$ see  theorem 18.6, \cite{Pierce},
which turns to be equal to the least common multiple of the orders of all the local invariants in $\mathbb{Q}/\mathbb{Z}$; see theorem 18.5, \cite{Pierce}. This fact, along with the Theorem \ref{dim}, gives the dimension of $A$.\\

The rest of section is dedicated to background useful for computing  supersingular Weil polynomials.\\ 

Let $U(\mathbb{Z} / m\mathbb{Z})$ denote the multiplicative group of integers modulo $m$ and $\zeta_m$ be a primitive $m$-th root of unity.
Then there exist a natural isomorphism of groups 
$$ U(\mathbb{Z} / m\mathbb{Z}) \rightarrow Gal(\mathbb{Q}(\zeta_{m})/ \mathbb{Q}), ~~~ a \rightarrow  \sigma_a$$
where $\sigma_a(\zeta_{m})=\zeta_{m}^a$.
There is a unique nontrivial homomorphism  $U(\mathbb{Z} / p\mathbb{Z})\rightarrow \{\pm 1 \}$ ($p$ an odd prime), the Legendre symbol $a \rightarrow (\frac{a}{p})$.
The group $U(\mathbb{Z} / 8\mathbb{Z})=\{\pm 1,  \pm 3 \}$ admits three nontrivial quadratic characters (homomorphisms  to $\{\pm 1 \}$):
\begin{enumerate}
\item $\varepsilon$, defined by $\varepsilon(a)\equiv a \mod 4$.
\item $\chi_{2}$, given by $\chi_{2}(\pm 1)=1, ~\chi_{2}(\pm 3)=-1$.
\item  $\chi_{-2}$, given by $\chi_{-2}(1)=\chi_{-2}(3)=1,\chi_{2}(-1)= ~\chi_{2}(-3)=-1$.
\end{enumerate}
For any odd prime $p$, $\sqrt{\pm p } \in \mathbb{Q}(\zeta_{4p})$ and 
$$ \displaystyle  \sigma_a(\sqrt{\pm p })=
                                               \left\{ \begin{array}{ll}
         (\frac{a}{p})\sqrt{\pm p } & \mbox{if $ \pm p \equiv 1 \mod 4$ };\\
         \varepsilon(a)(\frac{a}{p})\sqrt{\pm p }& \mbox{if $\pm p \equiv 3 \mod 4$ }.\end{array} \right. $$\\
         
         $\sqrt{\pm 2 } \in \mathbb{Q}(\zeta_{8})$ and for $a \in U(\mathbb{Z} / 8\mathbb{Z})$,\\
         $\sigma_a(\sqrt{2})= \chi_2(a)\sqrt{2}, ~ \sigma_a(\sqrt{-2})= \chi_{-2}(a)\sqrt{-2}$

For an odd prime $p$ we let 
\begin{displaymath}
p^{*}=(\frac{-1}{p})p= \left\{ \begin{array}{ll}
        p & \mbox{if $  p \equiv 1 \mod 4$ };\\
         -p & \mbox{if $p \equiv 3 \mod 4$ }.\end{array} \right. 
\end{displaymath}
For $p$ an odd prime and $t$ any odd number let \\
\begin{displaymath}
\displaystyle\Psi_{p,t}(X):=\prod_{a \in U(\mathbb{Z} / pt\mathbb{Z})}(X-(\frac{a}{p}) \zeta_{pt}^{a}) 
\end{displaymath}
\begin{displaymath}
\displaystyle\Psi_{2,t}(X):=\prod_{a \in U(\mathbb{Z} / t\mathbb{Z})}(X-\zeta_8\zeta_{t}^{a}) (X-\zeta_{8}^{-1}\zeta_{t}^{a})
\end{displaymath}
\begin{displaymath}
\displaystyle\Psi_{-2,t}(X):=\prod_{a \in U(\mathbb{Z} / t\mathbb{Z})}(X-\zeta_8\zeta_{t}^{a}) (X-\zeta_{8}^{3}\zeta_{t}^{a}).
\end{displaymath}

The polynomials below \\
$$\Psi_{p,t}(X) \in \mathbb{Q}(\sqrt{p^*})[X]$$
$$\Psi_{2,t}(X) \in \mathbb{Q}(\sqrt{2})[X]$$
$$\Psi_{-2,t}(X)  \in \mathbb{Q}(\sqrt{-2})[X]$$ factor as a product of 2 irreducibles over $\mathbb{Q}(\sqrt{p^*})[X], ~ \mathbb{Q}(\sqrt{2})[X], ~ \mathbb{Q}(\sqrt{-2})[X]$ respectively.
For example $\Psi_{p,t}(X)=\Psi_{p,t}^1(X)\Psi_{p,t}^{-1}(X)$ where $\Psi_{p,t}^{1}(X)=\prod_{\frac{a}{p}=1}(X-(\frac{a}{p}=1) \zeta_{pt}^{a})$ and $\Psi_{p,t}^{-1}(X)=\prod_{\frac{a}{p}=-1}(X+(\frac{a}{p}) \zeta_{pt}^{a}).$\\
If $K$ is a field, $f(X) \in K[X]$ and $a \in K^{\times}$, then  let $f^{[a]}(X):= a^{\deg f}f(\frac{X}{a})$.
Then $\Psi_{p,t}^{[\sqrt{p^*}]}(X)$,  $\Psi_{2,t}^{[\sqrt{2}]}(X)$,  $\Psi_{-2,t}^{[\sqrt{-2}]}(X) \in \mathbb{Q}[X]$ and are all irreducible over $\mathbb{Q}$.  
\section{Odd case}
\begin{theorem} Let $A$ be a supersingular simple abelian variety over $\mathbb{F}_q$ with  $q=p^r$ with $r$ odd. If the characteristic polynomial of $A$, $P_A(X)$, has no real root, then  $P_A(X)$ is irreducible.
\end{theorem}
\begin{proof}If $A$ is a supersingular abelian variety then it follows from the Theorem \ref{dim} that,  $inv_{\mathfrak{p}} (End_k (A)\otimes \mathbb{Q}) \equiv \frac{1}{2}e_if_i \mod \mathbb{Z}$, where $e_i$ is the ramification index of $\mathfrak{p}$ over $p$ and $f_i$ is the residue class degree.\\
But $e_i=e=v_{\mathfrak{p}}(p)=2v_{\mathfrak{p}}(\sqrt{p})$. This implies $inv_{\mathfrak{p}} (End_k (A)\otimes \mathbb{Q})= 0 \mod \mathbb{Z}$. In otherwords, order of invariants is zero and $2\dim A$ is equal to  the degree of the Weil polynomial. This along with the definition of the characteristic polynomial implies $P_A(X)$ is irreducible. 
\end{proof}
\begin{remark}  If $P(X)$ is a supersingular Weil polynomial with all real roots, over $\mathbb{F}_q$ with $q=p^n$, $n$ is odd, then $P(X)=X^2-q$, then we have least common multiple of local invariants as $2$, hence $P(X)^2$ is a characteristic polynomial of  simple supersingular abelian variety of dimension $2$, see \cite{KE}. 
\end{remark}
Let $p$ be a prime and $q=(\pm p)^r$ with $r$ odd. We wish to calculate the minimal polynomial over $\mathbb{Q}$ of $\sqrt{q}\zeta_m$.
Without lost of generality we assume that $m=4t$ for some $t$ since $\sqrt{q}\zeta_m= \sqrt{-q}\zeta_4\zeta_m= \sqrt{-q}\zeta_{m^{'}}$.
\begin{theorem}\label{odd}
Let $\theta= \sqrt{q}\zeta_{4t}$. We distinguish two cases.
\begin{enumerate}
\item Normal case If either
\begin{enumerate}
\item $q$ is odd and 
\begin{enumerate}
\item $t$ is even or 
\item $p \nmid t$ or 
\item $q \equiv 1 \mod 4$
\end{enumerate} or 
\item $q$ is even and $t \not\equiv 2 \mod 4$.
\end{enumerate} then the minimal polynomial of $\theta$ over $\mathbb{Q}$ is $\Phi_{4t}^{[\sqrt{q}]}(X)$.
\item exceptional case
\begin{enumerate}
\item If $q \equiv 3 \mod 4$ and $p|t$ and $t$ is odd, then the minimal polynomial of $\theta$ is  $\Psi_{p,\frac{t}{p}}^{[\sqrt{q^{*}}]}(X)$.
\item If $q >0$ is even and $t=2s \equiv 2 \mod 4$, then the minimal polynomial of $\theta$ is $\Psi_{2,s}^{[\sqrt{q}]}(X)$.
\item If $q <0$ is even and $t=2s \equiv 2 \mod 4$, then the minimal polynomial of $\theta$ is $\Psi_{-2,s}^{[\sqrt{q}]}(X)$.

\end{enumerate}
\end{enumerate}
\end{theorem}
\begin{proof}
Case (A)
First suppose $p \nmid t$. Then $\sqrt{q} \notin \mathbb{Q}(\zeta_{4t})$ and \\
$\Gamma:=Gal(\mathbb{Q}(\sqrt{q}, \zeta_{4t})/ \mathbb{Q})= <\sigma> \times Gal(\mathbb{Q}(\zeta_{4t})/ \mathbb{Q})$ where 
$\sigma(\sqrt{q})=-\sqrt{q}$, $\sigma(\zeta_{4t})=\zeta_{4t}$.\\

Let $\tau \in \Gamma$ and suppose $\tau(\theta)= \theta$. \\
´Case 1: $\tau =\sigma_a,~ a \in U(\mathbb{Z} /4 t\mathbb{Z})$ implies $\sqrt{q}\zeta_{4t}=\sqrt{q}\zeta_{4t}^{a}$ which implies $a =1$ which implies  $\tau = 1$.\\
Case 2: $\tau = \sigma \sigma_a$  implies $\sqrt{q}\zeta_{4t}=-\sqrt{q}\zeta_{4t}^{a}$ implies $\zeta_{4t}^{a}=-\zeta_{4t}=\zeta_{4t}^{1+2t}$  which implies $a =2t+1$.
So the subgroup of $\Gamma$ fixing $\theta$ is  $\{1, \sigma\sigma_{1+2t}\}$ and $Gal(\mathbb{Q}(\theta)/ \mathbb{Q})\cong Gal(\mathbb{Q}(\zeta_{4t})/ \mathbb{Q})$. 
In particular, the conjugates of $\theta$ over $\mathbb{Q}$ are $\sigma_{a}(\theta)=\sqrt{q}\zeta_{4t}^{a}$, and the minimal polynomial of $\theta$ over $\mathbb{Q}$ is $\displaystyle \prod_{a \in U(\mathbb{Z} / 4t\mathbb{Z})}(X-\sqrt{q}\zeta_{4t}^{a})=\Phi_{4t}^{[\sqrt{q}]}(X)$.\\
Case B.\\
 Now suppose $p$ is odd and $p|t$. 
Then $\sqrt{q} \in \mathbb{Q}(\zeta_{4t})$ and if $a \in U(\mathbb{Z} / 4t\mathbb{Z})$  we have \\
$$ \displaystyle \sigma_a(\theta) =  \sigma_a(\sqrt{q}\zeta_{4t})=
                                               \left\{ \begin{array}{ll}
         (\frac{a}{p})\sqrt{q}\zeta_{4t}^a & \mbox{if $q \equiv 1 \mod 4$ };\\
         \varepsilon(a)(\frac{a}{p})\sqrt{q}\zeta_{4t}^a& \mbox{if $q \equiv 3 \mod 4$ }.\end{array} \right. $$

Suppose $q \equiv  \mod 4 $ and $\sigma_a(\theta)=\theta$.
Then $(\frac{a}{p})\zeta_{4t}^a=\zeta_{4t}$.
If $(\frac{a}{p})=1$  then $a=1$ in $U(\mathbb{Z} / 4t\mathbb{Z})$ implies $\sigma_a=1$.
However  $(\frac{a}{p})=-1$ implies $\zeta_{4t}^a=\zeta_{4t}^{1+2t}$ which implies $a=1+2t$ or $a \equiv 1 \mod p$. This implies $(\frac{a}{p})=1$ which is a contradiction.\\
So $\sigma_a=1$ and  $Gal(\mathbb{Q}(\theta)/ \mathbb{Q})=(\mathbb{Q}(\zeta_{4t})/ \mathbb{Q})$  whence the minimal polynomial of $\theta$ is $\Phi_{4t}^{[\sqrt{q}]}(X)$.\\
Now suppose $q \equiv 3 \mod 4$ and $\sigma_a(\theta)=\theta$.
Then $\zeta_{4t}=\varepsilon(a)(\frac{a}{p})\sqrt{q}\zeta_{4t}^a$. Since $\varepsilon(a)(\frac{a}{p}) \in \{1, -1\}$ we have $a \in \{1, 1+2t\}$.
If $a=1+2t$ then $(\frac{a}{p})=1$ and $\varepsilon(a)= 1+2t \mod 4 =-1$ if and only if $t$ is odd.
Thus if $t$ is even, the minimal polynomial of $\theta$ is $\Phi_{4t}^{[\sqrt{q}]}(X)$.\\
However, if $t$ is odd, then $\theta$ is fixed by $\sigma_{1+2t}$ and hence the degree of its minimal polynomial over $\mathbb{Q}$ is $\frac{1}{2}\phi(4t)=\phi(t).$ However, in this case $\sqrt{q}\zeta_{4t}=\pm  \sqrt{-q}\zeta_{t}= \pm\sqrt{q^{*}}\zeta_{t} \in \mathbb{Q}(\zeta_t)$ and thus its minimal polynomial is $\prod_{a \in U(\mathbb{Z} / t\mathbb{Z})}(X-(\frac{a}{p}) \zeta_{t}^{a}) =\Psi_{p,\frac{t}{p}}^{[\sqrt{q^{*}}]}(X)$.
Case C:
$q$ is even and $2|t$. So $4t=8s$. Let $\zeta= \zeta_{4t}=\zeta_{8s}$. Then $\theta \in \mathbb{Q}{\zeta}$ and 
$$ \displaystyle \sigma_a(\theta) =  \sigma_a(\sqrt{q}\zeta)=
                                               \left\{ \begin{array}{ll}
         (\chi_{2}(a)\sqrt{q}\zeta^a & \mbox{if $q>0$ };\\
        \chi_{-2}(a)\sqrt{q}\zeta^a & \mbox{if $q<0$ }.\end{array} \right. $$

Suppose $\sigma_a(\theta)=\theta$. Then $\chi_{\pm 2}(a)=-1$ Then $a=1+4s \in  U(\mathbb{Z} / 8s \mathbb{Z})$. If $s$ is even then $a \equiv 1 \mod 8$ which implies $\chi_{\pm 2}(a)=1$ which a contradiction.
Thus if $s$ is even, the minimal polynomial of $\theta$ is $\Phi_{4t}^{[\sqrt{q}]}(X)$.\\
We can suppose that $4t=8s$ with $s$ odd. 
Then $1+4s=-3 \in U(\mathbb{Z} / 8s \mathbb{Z})$. It follows that $\theta$ is fixed by the $\{1, \sigma_{1+4s}\}$ and so its minimal polynomial has degree $\frac{1}{2}\phi(4t)$. We can write $\zeta$  as a product $\zeta_8 \zeta_s$. So $\theta= \sqrt{q}\zeta_8 \zeta_s$. \\
If $q>0$, the conjugates of $\theta$ are $\sqrt{q}\zeta_8 \zeta_s^a$ and $\sqrt{q}\zeta_8^{-1} \zeta_s^a$ as $a$ ranges over $U(\mathbb{Z} / s\mathbb{Z})$.
Thus the minimal polynomial is \\
$\displaystyle \prod_{a \in U(\mathbb{Z} / s\mathbb{Z})}(X-\sqrt{q}\zeta_8 \zeta_{s}^{a}) (X-\sqrt{q} \zeta_{8}^{-1}\zeta_{s}^{a})=\Psi_{2,s}^{[\sqrt{q}]}(X)$.\\
Similarly, if $q<0$, the minimal polynomial of $\theta$ is $\Psi_{-2,s}^{[\sqrt{q}]}(X)$.\\ This proves the theorem.
\end{proof}
\begin{remark}\label{R1}
\begin{enumerate}
\item In the normal case, the degree is $\phi(4t)$ and  in the exception case the degree is $\frac{1}{2}\phi(4t)$.

\item Since $\Psi_{p,t}^{[\sqrt{p^*}]}(X) \in \mathbb{Q}[X]$, it follows $\Psi_{p,t}(X) \in \mathbb{Q}(\sqrt{p^*})[X]$ and the coefficients of even powers are rational (even integers) and the coefficients of odd powers are integer multiples of $\sqrt{p^*}$. Similiar is true for $\Psi_{\pm 2,t}(X).$
\item If $p$ is odd, then $\Psi_{p,t}$ is a reciprocal polynomial when $p \equiv 1 \mod 4$ and anti-reciprocal (i.e; $ X^{\deg \Psi_{p,t} }\Psi_{p,t}(-\frac{1}{X}))=\Psi_{p,t}(X))$ when $p \equiv 3 \mod 4$.
\item $\Psi_{p,t}(X)$ acutally depends on the choice of $\zeta_{pt}$. If we replace it with  $\zeta_{pt}^a$ where $(\frac{a}{p})= -1$ then we get $\Psi_{p,t}(-X)$. Similiar is true for $\Psi_{\pm 2,t}(X).$
We have $$\Psi_{p,t}(X)\Psi_{p,t}(-X)=\Phi_{pt}(X^2).$$

\item It easily follows that polynomial $\Psi_{p,t}$ can be calculated recursively in the same way as cyclotomic polynomials.
$$\Psi_{p,tp^k}(X)=\Psi_{p,t}(X^{p^{k-1}}).$$ 
While if $l$ is a prime not dividing $t$ then 
$$\Psi_{p,l^k}(X)=\frac{\Psi_{p,t}(X^{l^k})}{\Psi_{p,t}(X^{l^{k-1}})}.$$

\end{enumerate}

\end{remark}
\section{Dimension 1}
\subsection{Normal Case} In this case $\phi(4t)= 2g = 2$ which implies $t=1$. Therefore  $X^2 \pm q$ is a Weil polynomial. 
\subsection{Exceptional Case} In this case $\frac{1}{2}\phi(4t)=2g=2$. That implies $4t \in \{8 ,12 \}$ or $t \in \{ 2, 3 \}$.
  In case  $t=2$  then $q=(\pm 2)^r$ with $r$ odd. Then,
  \begin{enumerate} 
  \item If $q>0$ then the minimal polynomial of $\theta$ is  $\Psi_{2,1}^{[\sqrt{q}]}(X)$.\\
But, $\Psi_{2,1}(X)=(X-\zeta_8) (X-\zeta_{8}^{-1})= X^2 \pm (\zeta_8+\zeta_8^{-1})X+1 = X^2\pm \sqrt{2}X+1$. Therefore 
   $\Psi_{2,1}^{[\sqrt{q}]}(X)= X^2\pm \sqrt{2q}X+q$.
  \item   If $q<0$, then the minimal polynomial of $\theta$ is  $\Psi_{-2,1}^{[\sqrt{q}]}(X)$.\\
But, $\Psi_{2,1}(X)=(X-\zeta_8) (X-\zeta_{8}^{3})= X^2 \pm (\zeta_8+\zeta_8^{3})X-1 = X^2\pm \sqrt{-2}X-1$. Therefore 
   $\Psi_{-2,1}^{[\sqrt{q}]}(X)= X^2\pm \sqrt{-2q}X-q$.
   \end{enumerate}
     In case $t=3$, then $p=3$ and $q=3^r$ with $r$ odd. In that case, the minimal polynomial of $\theta$ is  $\Psi_{3,1}^{[\sqrt{q}]}(X).$\\
But, $\displaystyle\Psi_{3,1}(X)  =\prod_{a \in U(\mathbb{Z} / 3\mathbb{Z})}(X-(\frac{a}{3}) \zeta_{3}^{a})= X^2- \sum_{a=1}^{2}(\frac{a}{p})X+1= X^2\pm \sqrt{-3}X-1$.\\
 Therefore $\Psi_{3,1}^{[\sqrt{q^*}]}(X)= X^2\pm \sqrt{3q}X+q$.\\

\section{Dimension 2}
\subsection{Normal Case} In this case $\phi(4t)= 2g = 4$. That implies $4t \in  \{8,12 \}$ or $t \in \{2, 3\}$.
Therefore either  $t=2$, $p$ odd and minimal polynomial is \\
$$\Phi_{8}^{[\sqrt{q}]}(X)=X^4+q^2,$$
or $t=3$, $q \neq 3^r$, $r ~odd$ and minimal polynomial is \\
$$\Phi_{12}^{[\sqrt{q}]}(X)=X^4-qX^2+q^2.$$
\subsection{Exceptional Case} In this case $\frac{1}{2}\phi(4t)=2g=4$. That implies $4t \in  \{16, 20, 24\}$ or $t \in \{4,5, 6\}$.
\begin{enumerate}
\item Either $p$ is odd, $t$ is odd, $p$ divides $t$, $q \equiv 3 \mod 4$.

In that case,  $t=p=5$, $q=(-5)^r$ with $r$ odd and $\Psi_{5,1}^{[\sqrt{q^*}]}(X)$ is the minimal polynomial of $\theta$.
We have $\Psi_{5,1}(X)\Psi_{5,1}(-X)=\Phi_5(X^2)$. Therefore,  $\Psi_{5,1}(X)= X^4 \pm  \sqrt{5}X^3 + 3X^2 \pm  \sqrt{5}X + 1$ and $$\Psi_{5,1}^{[\sqrt{q^*}]}(X)=X^4 \pm  \sqrt{-5q}X^3 - 3qX^2 \mp q\sqrt{-5q}X + q^2.$$

 \item $t \equiv 2 \mod 4, ~q=(\pm 2)^r$ with $r$ odd. 
 In that case,  $t=6=2\cdot 3 \mod 4$. Then 
 \begin{enumerate} 
  \item If $q>0$ then the minimal polynomial of $\theta$ is  $\Psi_{2,3}^{[\sqrt{q}]}(X)$.\\
But, $$\Psi_{2,3}(X)=\frac{\Psi_{2,1}(X^3)}{\Psi_{2,1}(X)}=X^4\pm \sqrt{2}X^3 + X^2 \pm \sqrt{2}X + 1$$ and hence
$$\Psi_{2,3}^{[\sqrt{q}]}(X)=X^4\pm \sqrt{2q}X^3 + qX^2 \pm q\sqrt{2q}X + q^2.$$
  \item   If $q<0$, then the minimal polynomial of $\theta$ is  $\Psi_{-2,3}^{[\sqrt{q}]}(X)$.\\
  But, $\Psi_{-2,3}(X)= X^4 \pm \sqrt{-2}X^3 - X^2 \mp \sqrt{-2}X + 1$ and hence 
  $\Psi_{-2,3}^{[\sqrt{q}]}(X)=X^4 \pm \sqrt{-2q}X^3 - qX^2 \mp q\sqrt{-2q}X + q^2$.
   \end{enumerate}
\end{enumerate}

\section{Dimension 3}
\subsection{Normal Case} In this case $\phi(4t)= 2g =6 $ and there is no such $t$.
\subsection{Exceptional Case} In this case $\frac{1}{2}\phi(4t)=2g=6$. That implies $4t \in  \{28, 36 \}$ or $t \in \{ 7, 9 \}$.
In that case, $p$ is odd, $t$ is odd, $p$ divides $t$, $q \equiv 3 \mod 4$. 
\begin{enumerate}
\item $t=7,~ p=7, ~q=7^r$ with $r$ odd, then $\Psi_{7,1}^{[\sqrt{q^*}]}(X)$ is the minimal polynomial of $\theta$.\\
But $\Psi_{7,1}(X)\Psi_{7,1}(-X)=\Phi_{7}(X^2)$ and therefore, $$\Psi_{7,1}(X)=X^6 \pm \sqrt{-7}X^5 - 3X^4 \mp \sqrt{-7}X^3 + 3X^2  \pm \sqrt{-7}X - 1.$$
This gives,
$$\Psi_{7,1}^{[\sqrt{q^*}]}(X)=X^6 \pm \sqrt{7q}X^5 + 3qX^4 \pm q\sqrt{7q}X^3 + 3q^2X^2  \pm q^2\sqrt{7q}X +q^3$$
\item $t=9,~ p=3, ~q=3^r$ with $r$ odd, then $\Psi_{3,3}^{[\sqrt{q^*}]}(X)$ is the minimal polynomial of $\theta$.\\
But, $$\Psi_{3,3}(X)=\Psi_{3,1}(X^3)=X^6 \pm  \sqrt{-3}X^3 - 1$$ and hence,
$$\Psi_{3,3}^{[\sqrt{q^*}]}(X)=X^6 \pm  q\sqrt{3q}X^3 +q^3.$$
\end{enumerate}
\section{Dimension 4}
\subsection{Normal Case} In this case $\phi(4t)= 2g =8 $. That implies $4t \in  \{16, 20, 24\}$ or $t \in \{4,5, 6\}$.
We have the following possibilities.
\begin{enumerate}
\item If $t=4$ then,  
\begin{enumerate}
\item $q$ is odd and $t$ is even,
\item $q$ is even and $t \not\equiv 2 \mod 4$. 
\end{enumerate}
Therefore the minimal polynomial is $$\Phi_{16}^{[\sqrt{q}]}(X)=X^8+q^4$$ for all primes $p$.

\item If $t=5$ then, 
\begin{enumerate}
\item $q$ is odd and $p \nmid t $ but $p \equiv 1 \mod 4$, this implies $q \neq (-5)^r,~ r$ odd.
\item $q$ is even and $t \not\equiv 2 \mod 4$. 
\end{enumerate}
 Therefore the minimal polynomial is $$\Phi_{20}^{[\sqrt{q}]}(X)=X^8-qX^6+q^2X^4-q^3X^2+q^4$$ for all $q\neq(-5)^r,~ r$ odd.  
 \item If $t=6$ then,
 \begin{enumerate}
\item $q$ is odd and $t$ is even,
\end{enumerate}
Therefore the minimal polynomial is $$\Phi_{24}^{[\sqrt{q}]}(X)=X^8-q^2X^4+q^4$$ for all primes $p \neq 2$.
\end{enumerate}
\subsection{Exceptional Case} In this case $\frac{1}{2}\phi(4t)=2g=8$. That implies $4t \in  \{32, 40, 48, 60\}$ or $t \in \{8, 10, 12, 15\}$. We have the following possibilities.

\begin{enumerate}
\item Either $p$ is odd, $t$ is odd, $p$ divides $t$, $q \equiv 3 \mod 4$.
In that case $t=15$ then 
\begin{enumerate}
\item either $p=3,~ q=3^r$ with $r$ odd and  $\Psi_{3,5}^{[\sqrt{q^*}]}(X)$ is the minimal polynomial of $\theta$.\\
Since $\Psi_{3,5}(X)=\frac{\Psi_{3,1}(X^5)}{\Psi_{3,1}(X)}$, we have $\Psi_{3,5}^{[\sqrt{q^*}]}(X)
=X^8 \mp \sqrt{3q}X^7 +2qX^6 \mp q\sqrt{3q}X^5 + q^2X^4 \mp\sqrt{3q}X^3 + 2X^2 \pm q^3\sqrt{3q}X + q^4$ or 

 \item $p=5,~q=(-5)^r$ with $r$ odd and  $\Psi_{5,3}^{[\sqrt{q^*}]}(X)$ is the minimal polynomial of $\theta$.\\
We have, $\Psi_{5,3}(X)=\frac{ \Psi_{5,1}(X^3)}{\Psi_{5,1}(X)}=X^8 \pm \sqrt{5}X^7 + 2X^6 \pm \sqrt{5}X^5 + 3X^4 \pm \sqrt{5}X^3 + 2X^2 \pm \sqrt{5}X + 1$ and hence $\Psi_{5,3}^{[\sqrt{q^*}]}(X)=
X^8 \pm \sqrt{-5q}X^7 - 2qX^6 \mp q \sqrt{-5q}X^5 + 3q^2X^4 \pm q^2 \sqrt{-5q}X^3 -2q^3X^2 \mp q^3\sqrt{-5q}X + q^4.$
\end{enumerate}
 \item $t \equiv 2 \mod 4, ~q=(\pm 2)^r$ with $r$ odd. In that case $t=10 =2 \cdot 5 \mod 4 $.
  \begin{enumerate}
  \item If $q>0$ then the minimal polynomial of $\theta$ is  $\Psi_{2,5}^{[\sqrt{q}]}(X)$.\\
But, $$\Psi_{2,5}(X)=\frac{\Psi_{2,1}(X^5)}{\Psi_{2,1}(X)}=X^8 \pm \sqrt{2}X^7 + X^6 - X^4 + X^2 \pm \sqrt{2}X + 1$$ and hence\\
$$\Psi_{2,5}^{[\sqrt{q}]}(X)=X^8 \pm \sqrt{2q}X^7 +q X^6 -q^2 X^4 + q^3X^2 \pm q^3\sqrt{2q}X + q^4 .$$
\item   If $q<0$, then the minimal polynomial of $\theta$ is  $\Psi_{-2,5}^{[\sqrt{q}]}(X)$.\\
  But, $\Psi_{-2,5}(X)=X^8 \pm \sqrt{-2}X^7 - X^6 - X^4 - X^2 \mp  \sqrt{-2}X + 1$. Therefore 
   $\Psi_{-2,5}^{[\sqrt{q}]}(X)= X^8 \pm \sqrt{-2q}X^7 -q X^6 -q^2 X^4 -q^3 X^2 \mp  q^3\sqrt{-2q}X + q^4$.\\ 
   \end{enumerate}

\end{enumerate}

\section{Dimension 5}
\subsection{Normal Case} In this case $\phi(4t)= 2g =10 $ and there is no such $t$.
\subsection{Exceptional Case} In this case $\frac{1}{2}\phi(4t)=2g=10$. That implies $4t =44$ or $t=11$. 
In that case, $p=11,~ q=11^r$, then $\Psi_{11,1}^{[\sqrt{q^*}]}(X)$ is the minimal polynomial of $\theta$. But $\Psi_{11,1}(X)\Psi_{11,1}(-X)= \Phi_{11}(X^2)$ which implies
$\Psi_{11,1}(X)=X^{10} \mp \sqrt{-11}X^9 - 5X^8  \pm \sqrt{-11}X^7 - X^6 \pm \sqrt{-11}X^5 + X^4 \pm \sqrt{-11}X^3 + 5X^2 \mp \sqrt{-11}X - 1$. Hence,
$\Psi_{11,1}^{[\sqrt{q^*}]}(X)=X^{10} \mp \sqrt{11q}X^9 +5qX^8  \mp q\sqrt{11q}X^7 -q^2 X^6 \pm q^2\sqrt{11q}X^5 -q^3X^4 \mp q^3\sqrt{11q}X^3 + 5q^4X^2 \mp q^4\sqrt{11q}X +q^5$.
\section{Dimension 6}
\subsection{Normal Case} In this case $\phi(4t)= 2g = 12$. That implies $4t \in  \{28, 36 \}$ or $t \in \{ 7, 9 \}$.
We have the following possibilities
\begin{enumerate}
\item If $t=7$ then,
\begin{enumerate}
\item $q$ is odd and $p\nmid t$ but $-7 \equiv 1 \mod 4$ i.e; $q\neq 7^r,~ r$ odd.
\item $q$ is even and $t \not\equiv 2 \mod 4$. 
\end{enumerate}
Therefore the minimal polynomial is $$\Phi_{28}^{[\sqrt{q}]}(X)=X^{12} - qX^{10} + q^2X^8 -q^3 X^6 + q^4X^4 -q^5 X^2 + q^6$$ for $q\neq 7^r,~ r$ odd.
\item If $t=9$ then,
\begin{enumerate}
\item $q$ is odd and $p\nmid t$ but $-3 \equiv 1 \mod 4$ i.e; $q\neq 3^r,~ r$ odd.
\item $q$ is even and $t \not\equiv 2 \mod 4$. 
\end{enumerate}
Therefore the minimal polynomial is $$\Phi_{36}^{[\sqrt{q}]}(X)=X^{12} - q^3X^6 + q^6 $$ for $q\neq 3^r,~ r$ odd.

\end{enumerate}

\subsection{Exceptional Case} In this case $\frac{1}{2}\phi(4t)=2g=12$. That implies $4t \in  \{52, 56, 72, 84\}$ or $t \in \{13, 14, 18, 21\}$.
\begin{enumerate}
\item Either $p$ is odd, $t$ is odd, $p$ divides $t$, $q \equiv 3 \mod 4$
\begin{enumerate}
\item $t=p=13$, $q=(-13)^r$ and $\Psi_{13,1}^{[\sqrt{q^*}]}(X)$ is the minimal polynomial of $\theta$.
But $\Psi_{13,1}(X)\Psi_{13,1}(-X)= \Phi_{13}(X^2)$ which implies  $\Psi_{13,1}(X)= X^{12} \pm \sqrt{13}X^{11} + 7X^{10} \pm 3\sqrt{13}X^9 + 15X^8 \pm 5\sqrt{13}X^7 + 19X^6 \pm 5\sqrt{13}X^5 + 15X^4
    \pm 3\sqrt{13}X^3 + 7X^2 \pm \sqrt{13}X + 1$ and therefore 
    $\Psi_{13,1}^{[\sqrt{q^*}]}(X)= X^{12} \pm \sqrt{-13q}X^{11} - 7qX^{10} \mp 3q\sqrt{-13q}X^9 + 15q^2X^8 \pm 5q^2 \sqrt{-13q}X^7-19q^3X^6 \mp 5q^3 \sqrt{-13q}X^5 + 15q^4X^4\pm 3q^4\sqrt{-13q}X^3 - 7q^5X^2 \mp q^5 \sqrt{-13q}X + q^6$.
\item $t=21$, then 
\begin{enumerate}
\item $p=3,~ q=3^r$ with $r$ odd and  $\Psi_{3,7}^{[\sqrt{q^*}]}(X)$ is the minimal polynomial of $\theta$.\\
But $ \Psi_{3,7}(X)=\frac{\Psi_{3,1}(X^7)}{\Psi_{3,1}(X)}=X^{12} \pm \sqrt{-3}X^{11} - 2X^{10} \mp \sqrt{-3}X^9 + X^8 + X^6 + X^4 \pm \sqrt{-3}X^3 - 2X^2 \mp \sqrt{-3}X + 1$\\ and hence
$ \Psi_{3,7}^{[\sqrt{q^*}]}(X)=X^{12} \pm \sqrt{3q}X^{11} + 2qX^{10} \pm q\sqrt{3q}X^9 + q^2X^8 -q^3 X^6 + q^4X^4 \pm q^4\sqrt{3q}X^3 + 2q^5X^2 \pm q^5\sqrt{3q}X + q^6$ or 
\item $p=7,~q=7^r$ with $r$ odd and  $\Psi_{7,3}^{[\sqrt{q^*}]}(X)$ is the minimal polynomial of $\theta$.\\
\end{enumerate}
$\Psi_{7,3}(X)=\frac{\Psi_{7,1}(X^3)}{\Psi_{7,1}(X)}=X^{12} \pm \sqrt{-7}X^{11} - 4X^{10} \mp \sqrt{-7}X^9 - X^8 \mp 2\sqrt{-7}X^7 + 7X^6 \pm 2\sqrt{-7}X^5 - X^4  \pm\sqrt{-7}X^3 -4X^2 \mp \sqrt{-7}X + 1.$ Hence, 
$\Psi_{7,3}^{[\sqrt{q^*}]}(X)=X^{12} \pm \sqrt{7q}X^{11} +4qX^{10} \pm q\sqrt{7q}X^9 - q^2X^8 \mp 2q^2\sqrt{7q}X^7 - 7q^3X^6 \mp  2q^3\sqrt{7q}X^5 -q^4 X^4  \pm q^4\sqrt{7q}X^3 +4q^5X^2 \pm q^5 \sqrt{7q}X + q^6$.

 \end{enumerate}
 \item $t \equiv 2 \mod 4, ~q$ even. 
 \begin{enumerate}
 \item $t=14=2\cdot 7 \mod 4$.

 \begin{enumerate}
  \item If $q>0$ then the minimal polynomial of $\theta$ is  $\Psi_{2,7}^{[\sqrt{q}]}(X)$.\\
But  $\Psi_{2,7}(X)=\frac{\Psi_{2,1}(X^7)}{\Psi_{2,1}(X)}=X^{12} \mp \sqrt{2}X^{11} + X^{10} - X^8 \pm \sqrt{2}X^7 - X^6 \pm \sqrt{2}X^5 - X^4 + X^2 \mp \sqrt{2}X + 1$ and hence,
 $\Psi_{2,7}^{[\sqrt{q}]}(X)=X^{12} \mp \sqrt{2q}X^{11} + qX^{10} -q^2 X^8 \pm q^2\sqrt{2q}X^7 - q^3X^6 \pm q^3\sqrt{2q}X^5 - q^4X^4 + q^5X^2 \mp q^5\sqrt{2q}X + q^6$.
\item   If $q<0$, then the minimal polynomial of $\theta$ is  $\Psi_{-2,7}^{[\sqrt{q}]}(X)$.\\
  But, $\Psi_{-2,7}(X)=X^{12} \pm \sqrt{-2}X^{11} - X^{10} - X^8 \mp \sqrt{-2}X^7 + X^6 \pm \sqrt{-2}X^5 - X^4 - X^2 \mp \sqrt{-2}X + 1$

$\Psi_{-2,7}^{[\sqrt{q}]}(X)=X^{12} \pm \sqrt{-2q}X^{11} -q X^{10} -q^2 X^8 \mp q^2 \sqrt{-2q}X^7 + q^3X^6 \pm q^3\sqrt{-2q}X^5 -q^4 X^4 -q^5 X^2 \mp q^5\sqrt{-2q}X + q^6$
 \end{enumerate}
 
 \item $t=18=2\cdot 9 \mod 4$ with $q$ even. In that case, 
  \begin{enumerate}
  \item If $q>0$ then the minimal polynomial of $\theta$ is  $\Psi_{2,9}^{[\sqrt{q}]}(X)$.\\
But, 
 $$\Psi_{2,9}(X)= \frac{\Psi_{2,1}(X^{3^2})}{\Psi_{2,1}(X^3)}= X^{12} \pm \sqrt{2}X^9 + X^6 \pm \sqrt{2}X^3 + 1.$$ and hence,
  $$\Psi_{2,9}^{[\sqrt{q}]}(X)= X^{12} \pm q\sqrt{2q}X^9 +q^3 X^6 \pm q^4\sqrt{2q}X^3 + q^6.$$
 
\item   If $q<0$, then the minimal polynomial of $\theta$ is  $\Psi_{-2,9}^{[\sqrt{q}]}(X)$.\\
  But, $\Psi_{-2,9}(X)= X^{12} \pm \sqrt{-2}X^9 - X^6  \mp \sqrt{-2}X^3 + 1$. Therefore 
   $\Psi_{-2,9}^{[\sqrt{q}]}(X)= X^{12} \pm q\sqrt{-2q}X^9 -q^3 X^6 \pm q^4\sqrt{-2q}X^3 + q^6.$ 
   \end{enumerate}
 \end{enumerate}
\end{enumerate}
\section{Dimension 7}
\subsection{Normal Case} In this case $\phi(4t)= 2g = 14$ and there is no such $t$. 
\subsection{Exceptional Case} In this case $\frac{1}{2}\phi(4t)=2g=14$ and there is no such $t$.

\section{Even Case}
If $P(X)$ is a supersingular Weil polynomial with all real roots, over $\mathbb{F}_q$ with $q=p^n$, $n$ is even, then $P(X)=X\pm \sqrt{q}$, then we have least common multiple of local invariants as $2$, hence $P(X)^2$ is a characteristic polynomial of  simple supersingular abelian variety of dimension $1$. 
We have the following Theorem which states all the possible characteristic polynomials of abelian varieties over $\mathbb{F}_q$.
\begin{theorem}\label{EAV}
Let $\phi(m)=2g$. Then $(\Phi_m^{[\sqrt{q}]}(X))^e$ is a
characteristic polynomial of Frobenius endomorphism of a simple supersingular abelian
variety of dimension $ge$ over $\mathbb{F}_q,~ q=p^n, ~n$ even,
 where 
$$r= \left\{ \begin{array}{ll}
        \mbox{order of  $p$ in the multiplicative group $U(\mathbb{Z}/ m\mathbb{Z})$ }, & \mbox{if $(p,m)=1$};\\
         \mbox{$f(p^k-p^{k-1})$, $f$ is order of  $p$ in the multiplicative group $U(\mathbb{Z}/ s\mathbb{Z})$} & \mbox{if $m=p^ks$}.\end{array} \right.$$
and $$e=\left\{ \begin{array}{ll}
        \mbox{$1$}, & \mbox{if $r$ is even};\\
         \mbox{$2$}, & \mbox{if $r$ is odd}.\end{array} \right.$$

\end{theorem}
\begin{proof}Let $P(X)$ be an irreducible supersingular Weil-$q$-polynomial of degree $2g$.
Since $q$ is an even power of prime, $ \sqrt{q} \in \mathbb{Z}$. Using
Honda-Tate theorem  \ref{H-T}, we get $\frac{1}{q^g}(P(\sqrt{q}X))$ is an irreducible cyclotomic polynomial $\Phi_m$ over $\mathbb{Z}$ of degree $2g$, i.e. $\phi(m)=2g$.
Therefore $P(X)=\Phi_m^{[\sqrt{q}]}(X)$ is an irreducible supersingular Weil polynomial. 
To determine the dimension of the corresponding abelian variety to $P(X)$, we will need to factorise $P(X)$ over $\mathbb{Q}_p$.
If $P (X) = \displaystyle\prod_{i}f_i (X)$ is the factorisation over $\mathbb{Q}_p [X]$, then\\

$inv_{\mathfrak{p}_i}(End_k (A)\otimes \mathbb{Q})\displaystyle \equiv \frac{v_p (f_i (0))}{ v_p (q)}  \mod \mathbb{Z}  \equiv \frac{\deg f_i}{2} \mod \mathbb{Z}$.\\
Since  $P(X)=\Phi_m^{[\sqrt{q}]}(X)$,  $inv_{\mathfrak{p}_i} (End_k (A)\otimes \mathbb{Q})\displaystyle = \frac{\deg r_i}{2} \mod \mathbb{Z}$, where $\Phi_m(X) = \displaystyle\prod_{i}r_i (X)$ over $\mathbb{Q}_p$. But it follows from the chapter IV.4 in \cite{Serre} that $\deg r_i= r $,  where 
$$r= \left\{ \begin{array}{ll}
        \mbox{order of  $p$ in the multiplicative group $U(\mathbb{Z}/ m\mathbb{Z})$ }, & \mbox{if $(p,m)=1$};\\
         \mbox{$f(p^k-p^{k-1})$, where $f$ is order of  $p$ in the multiplicative group $U(\mathbb{Z}/s\mathbb{Z})$,} & \mbox{if $m=p^ks$}.\end{array} \right.$$
Hence
$$ inv_{\mathfrak{p}_i} (End_k (A)\otimes \mathbb{Q}) = 
                                               \left\{ \begin{array}{ll}
         1 & \mbox{if $r$ is even};\\
         2 & \mbox{if $r$ is odd}.\end{array} \right. $$

 $~~~~~~~~~~~~~~~~~~~~~~~~~~~~~~~~~~~~~~~~~~~~~~=~e$ .\\ 
From Theorem \ref{dim}, $2 \dim A=e \deg P(X)$ or $\dim A=eg$.
\end{proof}
\section{Summary of results for dimensions $1$ to $7$}
We gather the list of  characteristic polynomials, from  dimensions $1$ to  $7$.
\begin{theorem}
Let $A$ be an supersingular simple abelian variety over $\mathbb{F}_q$, where $q=p^n$, $n$ odd. Then characteristic polynomial of $A$  is given by \\
\begin{enumerate}
\item{Dimension 1 \emph{(Deuring and Waterhouse\cite{deuring1941typen, Waterhouse})}}
\begin{enumerate}

\item $p=2:X^2 \pm\sqrt{2q}X+q$,
\item $p=3:X^2\pm \sqrt{3q}X+q$,
\item $X^2+q$.

\end{enumerate}
\item{Dimension 2 \emph{(C.Xing, D.Maisner and E.Nart\cite{Xing, Maiser.Nart})}}
\begin{enumerate}
\item $p\neq 3: X^4-qX^2+q^2$,
\item $X^4 + qX^2+q^2$,
\item $p=2:X^4\pm \sqrt{pq}X^3+qX^2\pm q \sqrt{pq}X+q^2$,
\item $p=5:X^4\pm \sqrt{pq}X^3+3qX^2\pm q \sqrt{pq}X+q^2$,
\item $(X^2-q)^2$,
\item $p \neq 2:X^4+q^2$.
\end{enumerate}

\item{Dimension 3  \emph{(E.Nart, C.Ritzenthaler and S.Haloui \cite{haloui2010characteristic, Nart})}}

\begin{enumerate}
\item $p=3: X^6 \pm q\sqrt{pq}X^3+q^3$, or
\item $p=7:X^6\pm\sqrt{pq}X^5+3qX^4\pm q\sqrt{pq}X^3+3q^2X^2\pm q^2\sqrt{pq}X+q^3$.
\end{enumerate}
\item{Dimension 4 \emph{( S.Haloui, V.Singh \cite{VSaf})}}
\begin{enumerate}
\item $X^8+q^4$,
\item $X^8- qX^6+q^2X^4- q^3X^2+q^4$,
\item $ p \neq 5: X^8+ qX^6+q^2X^4+q^3X^2+q^4$,
\item $p \neq 2 : X^8-q^2X^4+q^4$,
\item $p=3: X^8\pm \sqrt{3q}X^7+2qX^6\pm q \sqrt{3q}X^5+q^2X^4\pm q^2 \sqrt{3q}X^3+2q^3X^2\pm q^3\sqrt{3q}X+q^4$,
\item $p=5: X^8\pm \sqrt{5q}X^7+2qX^6\pm q \sqrt{5q}X^5+3q^2X^4\pm q^2 \sqrt{5q}X^3+2q^3X^2\pm q^3\sqrt{5q}X+q^4$,
\item $p=2: X^8\pm \sqrt{2q}X^7+qX^6-q^2X^4+q^3X^2 \pm q^3\sqrt{2q}X+q^4$.

\end{enumerate}
\item{Dimension 5}
\begin{enumerate}
\item $X^{10} \mp \sqrt{11q}X^9 +5qX^8  \mp q\sqrt{11q}X^7 -q^2 X^6 \pm q^2\sqrt{11q}X^5 -q^3X^4 \mp q^3\sqrt{11q}X^3 + 5q^4X^2 \mp q^4\sqrt{11q}X +q^5$.
\end{enumerate}
\item{Dimension 6}
\begin{enumerate}
\item  $X^{12} + qX^{10}+q^2X^8 + q^3X^6+q^4X^4 + q^5X^2+q^6$,
\item $p \neq 7: X^{12} - qX^{10}+q^2X^8 - q^3X^6+q^4X^4 - q^5X^2+q^6$,
\item $ X^{12}+q^3X^6+q^6$,
\item $p \neq 3 :  X^{12}-q^3X^6+q^6$,

\item $X^{12} \pm \sqrt{13q}X^{11} + 7qX^{10} \pm  3q\sqrt{13q}X^9 + 15q^2X^8 \pm 5q^2 \sqrt{13q}X^7+19q^3X^6 \pm 5q^3 \sqrt{13q}X^5 + 15q^4X^4\pm 3q^4\sqrt{13q}X^3 + 7q^5X^2 \pm  q^5 \sqrt{13q}X + q^6$

\item $X^{12} \pm \sqrt{3q}X^{11} + 2qX^{10} \pm q\sqrt{3q}X^9 + q^2X^8 -q^3 X^6 + q^4X^4 \pm q^4\sqrt{3q}X^3 + 2q^5X^2 \pm q^5\sqrt{3q}X + q^6.$

\item $X^{12} \pm \sqrt{7q}X^{11} +4qX^{10} \pm q\sqrt{7q}X^9 - q^2X^8 \mp 2q^2\sqrt{7q}X^7 - 7q^3X^6 \mp  2q^3\sqrt{7q}X^5 -q^4 X^4  \pm q^4\sqrt{7q}X^3 +4q^5X^2 \pm q^5 \sqrt{7q}X + q^6$,
 
\item $X^{12} \mp \sqrt{2q}X^{11} + qX^{10} -q^2 X^8 \pm q^2\sqrt{2q}X^7 - q^3X^6 \pm q^3\sqrt{2q}X^5 - q^4X^4 + q^5X^2 \mp q^5\sqrt{2q}X + q^6$,
\item $X^{12} \pm q\sqrt{2q}X^9 +q^3 X^6 \pm q^4\sqrt{2q}X^3 + q^6$.
\end{enumerate}
\item{Dimension 7}\\
There is no supersingular simple abelian variety of dimension $7$.
\end{enumerate}
\end{theorem}

\begin{theorem}
Let $A$ be an supersingular simple abelian variety over $\mathbb{F}_q$, where $q=p^n$, $n$ even. Then characteristic polynomial of $A$  is given by \\
\begin{enumerate}
\item{Dimension 1 \emph{(Deuring and Waterhouse\cite{deuring1941typen, Waterhouse})}}
\begin{enumerate}
\item $X^2+X\sqrt{q}+q $, $~~p \not\equiv 1 \mod  3$,
\item $X^2+q$,             $~~p\not\equiv  1 \mod 4$,
\item $X^2-X\sqrt{q}+q  $ , $~~ p \not\equiv  1 \mod 6$,
\item $(X\pm \sqrt{q})^2$.
\end{enumerate}
\item{Dimension 2 \emph{(C.Xing, D.Maisner and E.Nart\cite{Xing, Maiser.Nart})}}
\begin{enumerate}
\item $(X^2+X\sqrt{q}+q)^2 $, $~~p \equiv 1 \mod  3$,
\item $(X^2+q)^2$,             $~~p \equiv  1 \mod 4$,
\item $(X^2-X\sqrt{q}+q)^2 $ , $~~p \equiv  1 \mod 6$,
\item $X^4+\sqrt{q}X^3+qX^2+q^{3/2}X+q^2$,  $~~p \not\equiv 1 \mod 5$,
\item $X^4+q^2$,                 $~~p \not\equiv 1 \mod  8$,
\item $X^4-\sqrt{q}X^3+qX^2-q^{3/2}X+q^2$, $~~p \not\equiv 1 \mod   10$,
\item $X^4-qX^2+q^2$,  $~~p \not\equiv 1 \mod 12$.
\end{enumerate}

\item{Dimension 3  \emph{(E.Nart, C.Ritzenthaler and S.Haloui \cite{haloui2010characteristic, Nart})}}
\begin{enumerate}
\item $X^6+\sqrt{q}X^5+qX^4+q^{3/2}X^3+q^2X^2+q^{5/2}X+q^3$, $~~p \not\equiv 1,~2,~4 \mod  7$,
\item $X^6+q^{3/2}X^3+q^3$,    $~~p \not\equiv 1~,~4,~7 \mod  9$,
\item $X^6-\sqrt{q}X^5+qX^4-q^{3/2}X^3+q^2X^2-q^{5/2}X+q^3$,  $~~p \not\equiv 1,~9,~11 \mod 14$,
\item $X^6-q^{3/2}X^3+q^3$,  $~~p \not\equiv 1,7,13 \mod 18$.
\end{enumerate}
\item{Dimension 4 \emph{( S.Haloui, V.Singh \cite{VSaf})}}
\begin{enumerate}
\item $(X^4+\sqrt{q}X^3+qX^2+q^{3/2}X+q^2)^2$,  $~~p \equiv 1 \mod 5$,
\item $(X^4+q^2)^2$,                 $~~p \equiv 1 \mod  8$,
\item $(X^4-\sqrt{q}X^3+qX^2-q^{3/2}X+q^2)^2$, $~~p \equiv 1 \mod   10$,
\item $(X^4-qX^2+q^2)^2$,  $~~p \equiv 1 \mod 12$,
\item $X^8-\sqrt{q}X^7+q^{3/2}X^5-q^2X^4+q^{5/2}X^3-q^{7/2}X+q^4$, $p \not\equiv 1 \mod 15$,
\item $X^8+q^4$, $p \not\equiv 1 \mod 16$,
\item $X^8-qX^6+q^2X^4-q^3X^2+q^4$, $p \not\equiv 1 \mod 20$,
\item $X^8-q^2X^4+q^4$, $p \not\equiv 1 \mod 24$,
\item $X^8+\sqrt{q}X^7-q^{3/2}X^5-q^2X^4-q^{5/2}X^3+q^{7/2}X+q^4$, $p \not\equiv 1 \mod 30$.
\end{enumerate}
\item{Dimension 5}
\begin{enumerate}
\item $X^{10}+\sqrt{q}X^9+qX^8+q^{3/2}X^7+q^2X^6+q^{5/2}X^5+q^3X^4+q^{7/2}X^3+q^4X^2+q^{9/2}X+q^5$, 
$~~p \not\equiv 1,3,4,5,9 \mod 11$,
\item $X^{10}-\sqrt{q}X^9+qX^8-q^{3/2}X^7+q^2X^6-q^{5/2}X^5+q^3X^4-q^{7/2}X^3+q^4X^2-q^{9/2}X+q^5$,
$~~p \not\equiv 1,3,5,9,15 \mod 22$.
\end{enumerate}
\item{Dimension 6}
\begin{enumerate}
\item $(X^6+\sqrt{q}X^5+qX^4+q^{3/2}X^3+q^2X^2+q^{5/2}X+q^3)^2$, $~~p \equiv 1~,~2,~4 \mod  7$,
\item $(X^6+q^{3/2}X^3+q^3)^2$,    $~~p \equiv 1~,~4,~7 \mod  9$, 
\item $(X^6-\sqrt{q}X^5+qX^4-q^{3/2}X^3+q^2X^2-q^{5/2}X+q^3)^2$,  $~~p \equiv 1~,~9,~11 \mod 14$,
\item $(X^6-q^{3/2}X^3+q^3)^2$,  $~~p \equiv 1,7,13 \mod 18$,
\item $X^{12}+\sqrt{q}X^{11}+qX^{10}+q^{3/2}X^9+q^2X^8+q^{5/2}X^7+q^3X^6+q^{7/2}X^5+q^4X^4+q^{9/2}X^3+q^5X^2+q^{11/2}X+q^6$,
$~~p \not\equiv 1,3,9 \mod 13$,
\item $X^{12}-\sqrt{q}X^{11}+q^{3/2}X^9-q^2X^8+q^3X^6-q^4X^4+q^{9/2}X^3-q^{11/2}X+q^6$,
 $~~p \not\equiv 1,4,16 \mod 21$,
\item $X^{12}-\sqrt{q}X^{11}+qX^{10}-q^{3/2}X^9+q^2X^8-q^{5/2}X^7+q^3X^6-q^{7/2}X^5+q^4X^4-q^{9/2}X^3+q^5X^2-q^{11/2}X+q^6$,
$~~p \not\equiv 1, 3, 9 \mod 26$,
\item $X^{12}-qX^{10}+q^2X^8-q^3X^6+q^4X^4-q^5X^2+q^6$, $~~p \not\equiv 1,9, 25 \mod 28$,
\item $X^{12}-q^3X^6+q^6$,$~~p \not\equiv 1, 13, 25 \mod 36$,
\item $X^{12}+\sqrt{q}X^{11}-q^{3/2}X^9-q^2X^8+q^3X^6-q^4X^4-q^{9/2}X^3+q^{11/2}X+q^6$,
$~~p \not\equiv 1, 25, 37 \mod 42$.
\end{enumerate}
\item{Dimension 7}\\
There is no supersingular simple abelian variety of dimension $7$.
\end{enumerate}
\end{theorem}
\begin{proof}
The proof of this theorem, follows from the straightforward calculations using Theorem \ref{EAV}.
\end{proof}
\section{Existence of supersingular abelian varieties}
There are two important existential questions.
\begin{enumerate}
\item Given a positive integer $d$ and $q=p^n$, does there exist a simple supersingular abelian variety of dimension $d$ over $\mathbb{F}_q$?
\item How many isogeny classes of simple supersingular abelian varieties are there over  $\mathbb{F}_q$ for a  given dimension? 
\end{enumerate}
In this section, we will give partial answers to these questions.

\begin{theorem}
Let $g>2$ be an odd positive integer and $q=p^n, ~n$ even. Then the  characteristic polynomial of a supersingular simple abelian variety over $\mathbb{F}_q$ dimension $g$ is irreducible.
\end{theorem} 
\begin{proof} 
The characteristic polynomials of supersingular simple abelian variety of dimension $g$ are $P(X)^e$, where $e=1$ or $e=2$. When $e=1$, $P(X)$ is a  supersingular Weil polynomial of degree $2g$ and when $e=2$, $P(X)$ is a supersingular Weil polynomial of degree $g$.  Also all non linear Weil polynomials have even degree. Since $g>2$ and $g$ is odd, $e=2$ is not possible  and hence the theorem follows.
\end{proof}

\begin{theorem}\label{Ex1} Let $g>2$. Then there is no simple supersingular abelian  variety over $\mathbb{F}_q$, with $q=p^n$, $n$ even, of dimension $g$ if and only if $\phi^{-1}(g)$  and $\phi^{-1}(2g)$ are empty sets.
\end{theorem}
\begin{proof} A supersingular irreducible Weil polynomial of degree $2g$ is given by $\Phi_m^{[\sqrt{q}]}(X)$, where $\phi(m)=2g$. The characteristic polynomials of supersingular simple abelian variety of dimension $g$ are either irreducible supersingular Weil polynomials of degree $2g$ or the square of an irreducible supersingular Weil polynomials of degree $g$. There are no supersingular irreducible Weil polynomials of degree $g$ and $2g$ if  $\phi^{-1}(g)$  and $\phi^{-1}(2g)$ are empty sets, respectively. 
Conversely let $\phi^{-1}(g)$  or $\phi^{-1}(2g)$ be non empty set. If $\phi^{-1}(g)$ is not empty and  $g>2$ this implies $g$ is even. If $\phi(m)=g$, then then by Theorem \ref{EAV} $(\Phi_m^{[\sqrt{q}]}(X))^e$ is a characteristic polynomial of simple supersingular abelian variety over $\mathbb{F}_q$ of dimension $\frac{ge}{2}$, where $e$  the order of $p \mod m +1$ i.e; where $e=1$ or $e=2$. For $p$ such that order of $p$ modulo $m$ is odd, 
 $(\Phi_m^{[\sqrt{q}]}(X))^2$ is a characteristic polynomial of simple supersingular abelian variety over $\mathbb{F}_q$ of dimension $g$.
Similarly if $\phi^{-1}(2g)$ is non empty, for all primes $p$ such that order of $p$ modulo $m$ is even,  and repeating the argument above, we get  $\Phi_m^{[\sqrt{q}]}(X)$ is characteristic polynomial of a simple supersingular abelian variety of dimension $g$. Hence the theorem follows.
\end{proof}
\begin{remark}\label{Inf}
Given an integer  $m>2$, there are infinitely many primes which have an even order modulo $m$ and there are infinitely many primes which have an odd order modulo $m$.
\end{remark}

\begin{theorem}\label{Ex2} Let $g>2$. Then there is no simple supersingular abelian variety over $\mathbb{F}_q$, with $q=p^n$, $n$ odd, of dimension $g$ if  $ \phi^{-1}(g)$  and $\phi^{-1}(2g)$ are all empty sets.
\end{theorem}
\begin{proof} 
Let $q=p^n$, $n$ odd. By Theorem \ref{odd} and Remark \ref{R1}, there are no irreducible supersingular Weil polynomial of degree  $2g$ if $\phi(4t)=2g$ with  and $\frac{1}{2}\phi(4t)=2g$ has no solution for $t$. 
If  $2|m$ then we have $\phi(2m)= 2\phi(m)$. Therefore, $\phi(4t)=2g$ and  $\phi(4t)=4g$ has solutions for $t$  if and only if $\phi(2t)=g$  $\phi(2t)=2g$ has solutions for $t$. Therefore  there is no irreducible supersingular Weil polynomial of degree $2g$ if $\phi^{-1}(2g)$ and $\phi^{-1}(g)$ has no solution. Since the characteristic polynomial of a  simple supersingular abelian variety of dimension $g>2$ is irreducible, the result follows.  

\end{proof}
Though the complete characterisation of values of inverse of Euler function is a difficult problem, the above two theorems provide the following partial result. 
\begin{corollary}\label{Exis1}
If $p$ is prime greater than $2$ such that $2p+1$ is not prime(i.e. $p$ is not a Sophie Germain prime) then is no simple supersingular abelian variety of dimension $p$ over finite fields.
\end{corollary}

Therefore it follows from Theorems \ref{Ex1} and \ref{Ex2}, that for the following dimensions $g \leq 100 $ there are no simple supersingular abelian varieties.\\
prime :     $7, 13,17, 19, 31, 37, 43,47,61, 67, 71,73,79, 97. $\\
composite: $ 25, 27 , 34, 38, 45, 57,62, 63, 76,77, 85, 87,91,93, 94, 95.$
\begin{theorem}\label{EXI1}
There are infinitely many positive integers $g$ such that there is no simple supersingular abelian variety of dimension $g$ over finite fields.
\end{theorem}
\begin{proof}
We will show the Theorem using the fact that by there are infinitely many primes satisfying Corollary \ref{Exis1}.
It is well known that there are infinitely many primes congruent to $5 \mod 6$ which implies $2p+1 \equiv 4 \mod 6$. Hence $2p+1$ in this family is not prime and the Theorem follows.
\end{proof}
The following result is in contrast to Theorem \ref{EXI1}.
\begin{theorem}
There are infinitely many positive integers $g$ such that there are simple supersingular abelian varieties of dimension $g$.
\end{theorem}
\begin{proof} There are infinitely many positive integers $g$ such that $ \phi^{-1}(g)$  or $\phi^{-1}(2g)$ are not all empty sets. From  Theorem \ref{Ex1} and Remark \ref{Inf}, there is a supersingular simple abelian variety of dimension $g$ over $\mathbb{F}_q$, where $q$ is a square for infinitely many primes $p$ and hence this result follows.
\end{proof}
\bigskip
Let $G_{q,g}$ denote the number of isogeny classes of supersingular simple abelian varieties of dimension $g$ over $\mathbb{F}_{q}$ and 
$A(m):=\#\{x| \phi(x)=m\}$.
\begin{theorem}
If $g>1$ and $q=p^n,~ n$ even, then
 $$G_{q,g}= A(2g)(o(p,2g)+1)+ A(g)(o(p,g)),$$
 where $o(p,k)=$ the order of $p \mod k$, taken$\mod 2$.
\end{theorem}
\begin{proof}
It follows easily from Theorem \ref{EAV}.
\end{proof}
\begin{theorem}
If$g>2$ $q=p^n,~ n$ odd, then
 $$G_{q,g}\leq((-1)^{g+1}+1)A(2g)+ 2\sum_{n \in A(2g)}~\omega(n),$$
 where $\omega(n)$  is the number of distinct prime factors of $n$.
\end{theorem}
\begin{proof}
The proof follows from counting the supersingular irreducible Weil polynomial of degree $2g$ both in exceptional and normal case from the Theorem \ref{odd}. 

\end{proof}

\section*{Acknowledgements}
We would like to thank Kevin Hutchison, H.G. R\"{u}ck and  Christophe Ritzenthaler for the useful suggestions.

\renewcommand{\bibname}{Bibliography}
\bibliographystyle{IEEEtran}
\bibliography{vijay_bib_data}

\end{document}